\documentclass[10pt,reqno,sumlimits]{amsart} 
\usepackage{amssymb,amscd, epsfig, mathrsfs, xypic, amsmath,amsthm, color, enumerate}
\input xy \xyoption{all}
\CompileMatrices
\usepackage[dvipdfmx]{hyperref}
\usepackage{eepic, xcolor}
\definecolor{grey}{rgb}{0.86, 0.86, 0.86}
\newcommand{\Ess}{{Ess}}

\newcommand{\ZZ}{{\mathbb Z}}

\newcommand{\frakG}{{\mathfrak G}}

\newcommand{\calF}{{\mathcal F}}
\newcommand{\calG}{{\mathcal G}}

\newcommand{\calT}{{\mathcal T}}

\newcommand{\scA}{{\mathscr A}}

\newcommand{\id}{{\operatorname{id}}}

\newcommand{\tG}{{\widetilde{G}}}
\newcommand{\FSVT}{{\calT}}
\newcommand{\val}{{val}}

\newtheorem{thm}{Theorem}[section]  
 
\newtheorem{lem}[thm]{Lemma}  
\newtheorem{prop}[thm]{Proposition} 
\newtheorem{df-pr}[thm]{Definition-Proposition}

\theoremstyle{definition} 
\newtheorem{defn}[thm]{Definition}
   
\newtheorem{rem}[thm]{Remark}
 
\newtheorem{exm}[thm]{Example}

\numberwithin{equation}{section} 
\begin{document}
\title{Factorial Flagged Grothendieck Polynomials}
\author{Tomoo Matsumura, Shogo Sugimoto}
\date{\today}

\maketitle
\begin{abstract}
We show that the \emph{factorial flagged Grothendieck polynomials} defined by flagged set-valued tableaux of Knutson--Miller--Yong \cite{KnutsonMillerYong} can be expressed by a Jacobi--Trudi type determinant formula, generalizing the work of Hudson--Matsumura \cite{HudsonMatsumura2}. In particular, we obtain an alternative proof of the tableau and determinant formulas of vexillary double Grothendieck polynomials, that have been obtained by Knutson--Miller--Yong \cite{KnutsonMillerYong} and Hudson--Matsumura \cite{HudsonMatsumura2} respectively. Moreover, we show that each factorial flagged Grothendieck polynomial can be obtained from a product of linear polynomials by applying $K$-theoretic divided difference operators. 

\end{abstract}

\section{Introduction}
The {\it (double) Grothendieck polynomials} were introduced by Lascoux and Sch\"utzenberger (\cite{Lascoux1}, \cite{LascouxSchutzenberger3}) and they represent the (torus-equivariant) $K$-theory classes of the structure sheaves of Schubert varieties in the flag varieties. Their combinatorial formula in terms of pipe dreams or rc graphs was obtained by Fomin--Kirillov  (\cite{GrothendieckFomin}, \cite{DoubleGrothendieckFomin}). When restricted to Grassmannian elements, or more generally, vexillary permutations, Knutson--Miller--Yong \cite{KnutsonMillerYong} expressed the associated double Grothendieck polynomials as \emph{factorial flagged Grothendieck polynomials} defined in terms of {flagged set-valued tableaux}. This can be regarded as a unification of the work of Wachs \cite{Wachs} and Chen--Li--Louck \cite{ChenLiLouck} on flagged tableaux and the work of Buch \cite{BuchLRrule} and McNamara \cite{McNamara} on set-valued tableaux. On the other hand, the first author, in the joint work \cite{HIMN} with Hudson, Ikeda and Naruse, obtained a determinant formula of the double Grothendieck polynomials associated to Grassmannian elements (cf. \cite{Matsumura2016}). Such explicit closed formula was later generalized also to the vexillary case in \cite{HudsonMatsumura2} and \cite{Anderson2016}. Motivated by these results, the first author \cite{MatsumuraFlagged} studied non-factorial (single) flagged Grothendieck polynomials in general and showed their determinant formula, beyond the ones given by vexillary permutations.

In this paper, we extend the results in \cite{MatsumuraFlagged} to the factorial (double) case.  Let $x=(x_i)_{i\in \ZZ_{>0}}$ and $b=(b_i)_{\ZZ_{>0}}$ be infinite sequences of variables and $\beta$ a formal variable. We denote $u\oplus v=u+v+\beta uv$ for variables $u$ and $v$. Following the work \cite{KnutsonMillerYong} and \cite{KnutsonMillerYong2} of Knutson--Miller--Yong, we define the {\it factorial flagged Grothendieck polynomial} associated to a partition $\lambda=(\lambda_1\geq \cdots \geq\lambda_r>0)$ with a \emph{flagging} $f=(0<f_1\leq \cdots \leq f_r)$ by
\[
G_{\lambda,f}(x|b):=\sum_{T\in \FSVT(\lambda,f)}\beta^{|T|-|\lambda|}[x|b]^T 
\]
where $\FSVT(\lambda,f)$ denotes the set of flagged set-valued tableaux of the flagged partition $(\lambda,f)$. Here we denoted
\[
[x|b]^T = \prod_{e\in T} \left(x_{\val(e)} \oplus b_{\val(e)-r(e)+c(e)}\right)
\]
where the product runs over all entries $e$ of $T$, $\val(e)$ denotes the numeric value of $e$, and $r(e)$ (resp. $c(e)$) denotes the row (resp. column) index of $e$. The main result of this paper is as follows.

\vspace{2mm}


\noindent{\bf Main Theorem} (Theorem \ref{mainthm}). For a flagged partition $(\lambda,f)$ of length $r$, we have
\[
G_{\lambda,f}(x|b)=\det\left(\sum_{s=0}^{\infty}\binom{i-j}{s}\beta^s\calG_{\lambda_i+j-i+s}^{[f_i,f_i+\lambda_i-i]}\right)_{1\leq i,j\leq r}
\]
where $\calG_m^{[p,q]}=\calG_m^{[p,q]}(x|b), m\in \ZZ, p,q \in \ZZ_{\geq 0},$ are polynomials in $x$ and $b$ defined by the generating function
\[
\sum_{m\in \ZZ} \calG_m^{[p,q]} u^m = \frac{1}{1+\beta u^{-1}}\prod_{1\leq i\leq p}\frac{1+\beta x_i}{1-x_iu} \prod_{1\leq i\leq q} (1+ (u+\beta) b_i).
\]

\vspace{2mm}

Our proof is a generalization of the one in \cite{MatsumuraFlagged} to factorial (double) case. We show that both the tableau and determinant expressions satisfy the same compatibility with divided difference operators, which allows us to identify two formulas by induction. In fact, our proof shows that the original argument by Wachs \cite{Wachs} for flagged Schur polynomials works in the factorial case. A similar argument was employed by the first author in order to show a tableau formula of double Grothendieck polynomials associated to $321$-avoiding permutations \cite{Matsumura321}. If we specialize the above determinant formula by setting $\beta=0$, we recover the result by Chen--Li--Louck \cite{ChenLiLouck} for flagged double Schur functions. Our proof is different from theirs which is based on the lattice-path interpretation of the tableau expression.

It is also worth stressing that our proof is completely self-contained. In particular, as a consequence our main theorem, we obtain a purely algebraic and combinatoric proof of the formula
\[
\calG_m^{[p, p+m-1]}(x|b) = \sum_{T\in \FSVT((m), (p))} [x|b]^T.
\]
To the best of the authors' knowledge, the proof of this formula available in the literature uses a geometric argument established in \cite{HIMN}.

As an application, we obtain an alternative proof of the tableau and determinant formulas of vexillary double Grothendieck polynomials obtained by Knutson--Miller--Yong \cite{KnutsonMillerYong} and Hudson--Matsumura \cite{HudsonMatsumura2} respectively. We also generalize it to arbitrary factorial flagged Grothendieck polynomials. Namely, we show that they can be obtained from a product of linear polynomials. The analogous results for (non-factorial) flagged Schur/Grothendieck polynomials were obtained in \cite{Wachs} and \cite{MatsumuraFlagged}.

\section{Flagged Grothendieck polynomials}
\subsection{Flagged Partitions and Their Tableaux}
A {\it partition} $\lambda$ of length $r$ is a weakly decreasing sequence of positive integers $(\lambda_1,\dots,\lambda_r)$. We identify $\lambda$ with its {\it Young diagram (Ferrers diagram} $\{(i,j) \ |\ 1\leq i\leq r, 1\leq j\leq r\}$, depicted as a left-aligned array of boxes such that in the $i$th row from the top there are exactly $\lambda_i$ boxes. Let $|\lambda|$ be the total number of boxes in the Young diagram of $\lambda$, {\it i.e.} $|\lambda|=\lambda_1+\cdots+\lambda_r$. A {\it flagging} $f$ of a partition of length $r$ is a weakly increasing sequence of positive integers $(f_1,\dots, f_r)$. We call the pair $(\lambda,f)$ a {\it flagged partition}. 

A {\it set-valued tableau} $T$ of shape $\lambda$ is an assignment of a finite subset of positive integers to each box of the Young diagram of $\lambda$,  satisfying
\begin{itemize}
\item weakly increasing in each row: $\max A \leq \min B$ if $A$ fills  the box $(i,j)$ and $B$ fills the box $(i,j+1)$ for $1\leq i\leq r$ and $1\leq j\leq \lambda_i-1$.
\item strictly increasing in each column: $\max A < \min B$ if $A$ fills  the box $(i,j)$ and $B$ fills the box $(i+1,j)$ for $1\leq i\leq r-1$ and $1\leq j\leq \lambda_{i+1}$.
\end{itemize}
An element $e$ of a subset assigned to a box of $T$ is called an {\it entry} of $T$ and denoted by $e\in T$. The total number of entries in $T$ is denoted by $|T|$. For each $e\in T$, let $\val(e)$ be its numeric value, $r(e)$ its row index, and $c(e)$ its column index.  A {\it flagged set-valued tableau} of a flagged partition $(\lambda,f)$ is a set-value tableau of shape $\lambda$ additionally satisfying that the sets assigned to the boxes of the $i$-th row  are subsets of $\{1,\dots, f_i\}$. Let $\FSVT(\lambda,f)$ be the set of all flagged set-valued tableaux of $(\lambda,f)$.

Let $x=(x_i)_{\ZZ_{>0}}$ and $b=(b_i)_{\ZZ_{>0}}$ be sets of infinitely many variables. Let $\ZZ[\beta]$ be the polynomial ring of a formal variable $\beta$ where we set $\deg \beta =-1$. Let $\ZZ[\beta][x,b]$ and $\ZZ[\beta][[x,b]]$ be the rings of polynomials and of formal power series in $x$ and $b$ respectively. For each tableau $T \in \FSVT(\lambda,f)$, we define
\[
[x|b]^T := \prod_{e\in T}(x_{val(e)}\oplus b_{val(e)-r(e)+c(e)})
\]
where we set $u\oplus v :=u+v+\beta uv$. Following Knutson--Miller--Yong (\cite{KnutsonMillerYong}, \cite{KnutsonMillerYong2}), we define the {\it factorial flagged Grothendieck polynomial} $G_{\lambda,f}(x|b)$ as follows.
\begin{defn}
For a flagged partition $(\lambda,f)$, we define 
\[
G_{\lambda,f}=G_{\lambda,f}(x|b)=\sum_{T\in FSVT(\lambda,f)}\beta^{|T|-|\lambda|}[x|b]^T.
\]
If $\lambda$ is an empty partition, we set $G_{\lambda,f}=1$, and if $f_1=0$, we set $G_{\lambda,f}=0$.
\end{defn}
\begin{exm}\label{exm2-1}
Let $\lambda=(3,1)$ and $f=(2,4)$. Then $\FSVT(\lambda,f)$ contains tableaux such as

\vspace{2mm}

\begin{center}
\setlength{\unitlength}{0.5mm}
\begin{picture}(38,20)
\linethickness{0.7pt}
\put(0,20){\line(1,0){30}}\put(0,10){\line(1,0){30}}\put(0,0){\line(1,0){10}}\put(0,20){\line(0,-1){20}}\put(10,20){\line(0,-1){20}}\put(20,20){\line(0,-1){10}}\put(30,20){\line(0,-1){10}}
\put(3,13){$1$}\put(13,13){$1$}\put(22,13){$12$}
\put(1,3){$23$}
\end{picture}
\begin{picture}(38,20)
\linethickness{0.7pt}
\put(0,20){\line(1,0){30}}\put(0,10){\line(1,0){30}}\put(0,0){\line(1,0){10}}\put(0,20){\line(0,-1){20}}\put(10,20){\line(0,-1){20}}\put(20,20){\line(0,-1){10}}\put(30,20){\line(0,-1){10}}
\put(1,13){$12$}\put(13,13){$2$}\put(23,13){$2$}
\put(1,3){$34$}
\end{picture}
\begin{picture}(38,20)
\linethickness{0.7pt}
\put(0,20){\line(1,0){30}}\put(0,10){\line(1,0){30}}\put(0,0){\line(1,0){10}}\put(0,20){\line(0,-1){20}}\put(10,20){\line(0,-1){20}}\put(20,20){\line(0,-1){10}}\put(30,20){\line(0,-1){10}}
\put(3,13){$1$}\put(11,13){$12$}\put(23,13){$2$}
\put(1,3){$23$}
\end{picture}
\begin{picture}(38,20)
\linethickness{0.7pt}
\put(0,20){\line(1,0){30}}\put(0,10){\line(1,0){30}}\put(0,0){\line(1,0){10}}\put(0,20){\line(0,-1){20}}\put(10,20){\line(0,-1){20}}\put(20,20){\line(0,-1){10}}\put(30,20){\line(0,-1){10}}
\put(3,13){$2$}\put(13,13){$2$}\put(23,13){$2$}
\put(3,3){$4$}
\end{picture}.
\end{center}
If we change $f$ to $f'=(2,3)$, then $\FSVT(\lambda,f')$ doesn't contain the second and forth tableaux.

If $T$ is the first tableau in the above list, $|T|=6$ and we have
\[
[x|b]^T = (x_1\oplus b_1)(x_1\oplus b_2)(x_1\oplus b_3)(x_2\oplus b_4)(x_2\oplus b_1)(x_3\oplus b_2).
\]
\end{exm}
\subsection{Divided difference operators and $\calG_m^{[p,q]}$}
\begin{defn}
For each positive integer $i$, define the {\it $K$-theoretic divided difference operator} $\pi_i$ by
\[
\pi_i(f)=\frac{(1+\beta x_{i+1})f-(1+\beta x_{i})s_i(f)}{x_i-x_{i+1}}
\]
for each $f\in \ZZ[\beta][[x,b]]$, where $s_i$ permutes $x_i$ and $x_{i+1}$.
\end{defn}
The following properties of $\pi_i$ are easy to check (cf. \cite[\S2.1]{MatsumuraFlagged}). 
\begin{lem}\label{lem1}
For each positive integer $i$ and $f,g\in \ZZ[\beta][[x,b]]$, we have
\begin{enumerate}[$(1)$]
\item $\pi_i(fg)=\pi_i(f)g+s_i(f)\pi_i(g)+\beta s_i(f)g$.
\item If $f$ is symmetric in $x_i$ and $x_{i+1}$,  then $\pi_i(f)=-\beta f$ and $\pi_t(fg)=f\pi_t(g)$.
\item $\pi_i(f)=-\beta f$, then $f$ is symmetric in $x_i$ and $x_{i+1}$.
\end{enumerate}
\end{lem}
\begin{defn}
Define $\calG_m^{[p,q]}=\calG_m^{[p,q]}(x|b), m\in \ZZ, p,q \in \ZZ_{\geq 0}$, by the generating function
\[
\sum_{m\in \ZZ} \calG_m^{[p,q]}u^m = \frac{1}{1+\beta u^{-1}}\prod_{1\leq i\leq p}\frac{1+\beta x_i}{1-x_iu} \prod_{1\leq i\leq q} (1+ (u+\beta) b_i).
\]
If $q=0$, then we denote $\calG_m^{[p]}=\calG_m^{[p,q]}$.
\end{defn}
\begin{rem}\label{rem1}
We can easily check by direct computation (cf. \cite{MatsumuraFlagged}) that 
\[
\calG_m^{[p,q]}=(-\beta)^{-m} \ (m\leq 0), \ \ \ \  \calG_m^{[1]}=x_1^m\ (m\geq 0), \ \ \ \mbox{and} \ \ \calG_m^{[0,q]} = 0 \  (m> q).
\]
Moreover, we have $ \calG_{q+r}^{[1,q]}=x_1^r \calG_{q}^{[1,q]}$ for any integer $r\geq 0$. Indeed, if we let
\[
\prod_{i=1}^{q}\left(1+(u+\beta) b_i\right)=\sum_{i=0}^qE_i u^i,
\]
then, 
\[
\calG_{q+r}^{[1,q]}=\sum_{i=0}^{d} \calG_{i+r}^{[1]}E_{q-i}=x_1^r \sum_{i=0}^{q}\calG_i^{[1]}E_{q-i}=x_1^r \calG_q^{[1,q]}.
\]
\end{rem}
The following basic lemmas will be used in the next section.
\begin{lem}\label{lem2}
For each $m\in \ZZ, p,q \in \ZZ_{\geq 0}$, we have
$\pi_i({\calG}_m^{[p,q]})=
\begin{cases}
\calG_{m-1}^{[p+1,q]}&(i=p),\\
-\beta \calG_m^{[p,q]}&(i\neq p).
\end{cases}
$
\end{lem}
\begin{proof}
The claim follows from applying divided difference operators to the generating function of $\calG_m^{[p,q]}$ (cf. \cite[Lemma 1]{MatsumuraFlagged}).
\end{proof}
\begin{lem}\label{lem3}
Let $t$ and $t_i \ (i=1,\dots, n)$ be arbitrary positive integers and $t':=t+1$. We have
\[
\pi_t\left(\prod_{i=1}^{n}(x_t\oplus b_{t_i})\right)
=\sum_{v=0}^{n-1}\left(\prod_{i=1}^v(x_t\oplus b_{t_i})\prod_{i=v+2}^{n}(x_{t'}\oplus b_{t_i})\right) +\beta\sum_{v=1}^{n-1}\left(\prod_{i=1}^v(x_t\oplus b_{t_i})\prod_{i=v+1}^n(x_{t'}\oplus b_{t_i})\right).
\]
In particular, $\pi_t(x_t\oplus b_{t_1})=1$.
\end{lem}
\begin{proof}
We prove the claim by induction on $n$. When $n=1$, we can show that $\pi_t(x_t\oplus b_s)=1$ by direct computation. For a general $n$, we apply Lemma \ref{lem1} (1) with $f=x_t\oplus b_{t_n}$ and $g$ the rest:
\begin{eqnarray*}
\pi_t\left(\prod_{i=1}^n(x_t\oplus b_{t_i})\right)
=\prod_{i=1}^{n-1}(x_t\oplus b_{t_i})
+(x_{t'}\oplus b_{t_n})\pi_t\left(\prod_{i=1}^{n-1}(x_t\oplus b_{t_i})\right)
+\beta (x_{t'}\oplus b_{t_n})\prod_{i=1}^{n-1}(x_t\oplus b_{t_i}).
\end{eqnarray*}
By the induction hypothesis, we have 
\begin{eqnarray*}
\pi_t\left(\prod_{i=1}^n(x_t\oplus b_{t_i})\right)\!\!\!&=&\!\!\!\prod_{i=1}^{n-1}(x_t\oplus b_{t_i})+(x_{t'}\oplus b_{t_n})\sum_{v=0}^{n-2}\left(\prod_{i=1}^v(x_t\oplus b_{t_i})\prod_{i=v+2}^{n-1}(x_{t'}\oplus b_{t_i})\right) \\
&&\!\!\!\!\!\!+\beta\left((x_{t'}\oplus b_{t_n})\sum_{v=1}^{n-2}\left(\prod_{i=1}^v(x_t\oplus b_{t_i})\prod_{i=v+1}^{n-1}(x_{t'}\oplus b_{t_i})\right)+(x_{t'}\oplus b_{t_n})\prod_{i=1}^{n-1}(x_t\oplus b_{t_i})\right)
\end{eqnarray*}
This is exactly the right hand side of the desired formula.
\end{proof}
\begin{rem}\label{rem2}
Since the left hand side of the formula in Lemma \ref{lem3} is symmetric in variables $b_{t_1},\dots, b_{t_n}$, we can conclude that the right hand side is symmetric in $x_t$ and $x_{t+1}$. 
\end{rem}
\begin{lem}[cf. Lemma 3 \cite{MatsumuraFlagged}]\label{lem4}
For any integers $m\in \ZZ$, $p\in \ZZ_{\geq 1}$ and $q \in \ZZ_{\geq 0}$, we have
\[
\frac{1}{1+x_1\beta}\left( \calG_{m}^{[p,q]}-x_1 \calG_{m-1}^{[p,q]}\right)=\left(\calG_{m}^{[p-1,q]}\right)^{\star}, 
\]
where $\star$ replaces each $x_i$ by $x_{i+1}$.
\end{lem}
\begin{proof}
It follows from the identity
\[
 \sum_{m\in\ZZ}\frac{1}{1+x_1\beta}\left( \calG_{m}^{[p,q]}-x_1 \calG_{m-1}^{[p,q]}\right)u^m
 =\left(\sum_{m\in\ZZ}\calG_{m}^{[p-1,q]}u^{m}\right)^{\star},
\]
which can be check by a direct computation.
\end{proof}
\section{Propositions and the Main theorem}
In this section, we prove the main theorem. First we show three propositions that will be used in the proof of the main theorem given at the end of this section. Through the section, we let $(\lambda,f)$ to be a flagged partition of length $r$.
\begin{defn}
We denote the following determinantal expression by $\tG_{\lambda,f}=\tG_{\lambda,f}(x|b)$:
\[
\tG_{\lambda,f}(x|b)=\det\left(\sum_{s=0}^{\infty}\binom{i-j}{s}\beta^s \calG_{\lambda_i-i+j+s}^{[f_i,f_i+\lambda_i-i]}\right)_{(1\leq i,j\leq r)}.
\]
If $\lambda$ is empty, we set $\tG_{\lambda,f}=1$. If $f_1=0$, the right hand side is identically zero by Remark \ref{rem1} so that we set $\tG_{\lambda,f}=0$.
\end{defn}
\begin{prop}\label{prop1}
For any integer $q\geq 0$, we have $\calG_{q}^{[1,q]}=\displaystyle\prod_{i=1}^q(x_1\oplus b_i)$. 
\end{prop}
\begin{proof}
We prove the claim by induction on $q$. When $q=0$, it is trivial. Suppose that $q>0$. By definition of $\calG_m^{[p,q]}$, we have
\[
\sum_{m\in\ZZ} \calG_m^{[1,q]}u^m=\left(\sum_{m\in\ZZ} \calG_m^{[1,q-1]}u^m\right)(1+\beta b_q+b_qu),
\]
so that 
\[
\calG_q^{[1,q]} = \calG_{q}^{[1,q-1]}(1+\beta b_q)+ \calG_{q-1}^{[1,q-1]}b_q.
\]
Since $\calG_{q}^{[1,q-1]}=x_1\calG_{q-1}^{[1,q-1]}$ by Remark \ref{rem1}, we obtain
\[
\calG_q^{[1,q]}=\calG_{q-1}^{[1,q-1]}x_1(1+\beta b_q)+ \calG_{q-1}^{[1,q-1]}b_q = \calG_{q-1}^{[1,q-1]} (x_1\oplus b_q).
\]
The induction hypothesis implies the desired formula. 
\end{proof}
\begin{prop}\label{prop2}
Suppose that $f_1=1$ and let 
\[
\lambda'=(\lambda_2,\lambda_3\cdots,\lambda_r)\ \ \  \mbox{and} \ \ \ f'=(f_2-1,f_3-1,\cdots,f_r-1).
\]
Then we have
\begin{itemize}
\item[(1)]$\tG_{\lambda,f}(x|b)= \calG_{\lambda_1}^{[1,\lambda_1]}\cdot \tG_{\lambda',f'}(x|b)^{\star}$,
\item[(2)]$G_{\lambda,f}(x|b)=\calG_{\lambda_1}^{[1,\lambda_1]}\cdot G_{\lambda',f'}(x|b)^{\star}$,
\end{itemize}
where $\star$ replaces each $x_i$ by $x_{i+1}$. 
\end{prop}
\begin{proof}
(1) We show that the left hand side coincides with the right hand side by the column operation of subtracting the $(j-1)$-st column multiplied with $x_1(1+x_1\beta)^{-1}$ from the $j$-th column for $j=2,\dots,r$. By Remark \ref{rem1}, the $(1,j)$ entry of $\tG_{\lambda,f}(x|b)$ is
\[
\sum_{s=0}^{\infty}\binom{1-j}{s}\beta^s \calG_{\lambda_1-1+j+s}^{[1,\lambda_1]}
=\sum_{s=0}^{\infty}\binom{1-j}{s}\beta^sx_1^{j-1+s}\calG_{\lambda_1}^{[1,\lambda_1]}
=(1+x_1\beta)^{1-j}x_1^{j-1}\calG_{\lambda_1}^{[1,\lambda_1]}.
\]
Thus after the above column operation, the first row of $\tG_{\lambda,f}(x|b)$ becomes $(\calG_{\lambda_1}^{[1,\lambda_1]},0,\dots, 0)$. We compute the $(i,j)$ entry of the resulting determinant for $i,j\geq 2$:
\begin{eqnarray*}
&&\sum_{s=0}^{\infty}
\binom{i-j}{s}
\beta^s \calG_{\lambda_1-i+j+s}^{[f_i,f_i+\lambda_i-i]}
-\frac{x_1}{1+x_1\beta}\sum_{s=0}^{\infty}\binom{i-j+1}{s}\beta^s \calG_{\lambda_i-i+j-1+s}^{[f_i,f_i+\lambda_i-i]}\\
&=&\quad\sum_{s=0}^{\infty}\binom{i-j}{s}\beta^s \calG_{\lambda_1-i+j+s}^{[f_i,f_i+\lambda_i-i]}
-\frac{x_1}{1+x_1\beta}\sum_{s=0}^{\infty}\binom{i-j}{s}
+\binom{i-j}{s-1}\beta^s \calG_{\lambda_i-i+j-1+s}^{[f_i,f_i+\lambda_i-i]}\\
&=&\sum_{s=0}^{\infty}\binom{i-j}{s}\beta^s\left( \calG_{\lambda_1-i+j+s}^{[f_i,f_i+\lambda_i-i]}
-\frac{x_1}{1+x_1\beta} \calG_{\lambda_1-i+j+s-1}^{[f_i,f_i+\lambda_i-i]}\right)
-\sum_{s'=0}^{\infty}\binom{i-j}{s'}\beta^{s'+1} \calG_{\lambda_i-i+j+s'+1}^{[f_i,f_i+\lambda_i-i]}\\
&=&\sum_{s=0}^{\infty}\binom{i-j}{s}\beta^s\frac{1}{1+x_1\beta}\left(\calG_{\lambda_i-i+j+s}^{[f_i,f_i+\lambda_i-i]}-x_1 \calG_{\lambda_i-i+j+s-1}^{[f_i,f_i+\lambda_i-i]}\right)\\
&=&\sum_{s=0}^{\infty}\binom{i-j}{s}\beta^s\left(\calG_{\lambda_i-1+j+s}^{[f_i-1,f_i+\lambda_i-i]}\right)^{\star}.
\end{eqnarray*}
Here we have used a well-known identity for binomial coefficients for the first equality, and Lemma \ref{lem4} for the last equality. Finally the desired formula follows from the cofactor expansion with respect to the first row for the resulting determinant after the column operation.

 (2) For any $T\in \FSVT(\lambda,f)$, the entries on the first row of $T$ are all $1$ and all other entries are greater than $1$. There is a bijection from $\FSVT(\lambda,f)$ to $\FSVT(\lambda',f')$ sending $T$ to $T'$ obtained from $T$ by deleting its first row and decreasing the numeric values of the rest of the entries by $1$. Under this bijection, we have
\[
[x|b]^T = \left([x|b]^{T'}\right)^{\star}\cdot \prod_{j=1}^{\lambda_1}(x_1\oplus b_j).
\] 
Now the claim follows from Proposition \ref{prop1}.
\end{proof}
\begin{prop}\label{prop3}
If $\lambda_1>\lambda_2$ and $f_1<f_2$, then
\begin{itemize}
\item[(1)] $\pi_{f_1}\tG_{\lambda,f}(x|b)=\tG_{\lambda',f'}(x|b)$,
\item[(2)] $\pi_{f_1}G_{\lambda,f}(x|b)=G_{\lambda',f'}(x|b)$,
\end{itemize}
where $\lambda'=(\lambda_1-1,\lambda_2,\cdots,\lambda_r)$ and $f'=(f_1+1,f_2,\cdots,f_r)$.
\end{prop}
\begin{proof}
(1) First observe that the entries of the determinant are symmetric in $x_{f_1}$ and $x_{f_1+1}$ except the ones on the first row, since $f_1<f_2$. We consider the cofactor expansion of $\tG_{\lambda,f}(x|b)$ with respect to the first row: let $\Delta_{i,j}$ be the cofactor of the $(i,j)$ entry and we have
\[
\tG_{\lambda,f} = \sum_{j=1}^r (-1)^{1+j}\Delta_{1,j}\sum_{s=0}^{\infty}\binom{1-j}{s}\beta^s \calG_{\lambda_1-1+j+s}^{[f_1,f_1+\lambda_1-1]}.
\]
Applying $\pi_{f_1}$ to this expansion, then we obtain
\begin{eqnarray*}
\pi_{f_1}\tG_{\lambda,f} 
&=&  \sum_{j=1}^r (-1)^{1+j}\Delta_{1,j}\sum_{s=0}^{\infty}\binom{1-j}{s}\beta^s\pi_{f_1}\left(\calG_{\lambda_1-1+j+s}^{[f_1,f_1+\lambda_1-1]}\right)\\
&=&  \sum_{j=1}^r (-1)^{1+j}\Delta_{1,j}\sum_{s=0}^{\infty}\binom{1-j}{s}\beta^s\calG_{\lambda_1-2+j+s}^{[f_1+1,f_1+\lambda_1-1]},
\end{eqnarray*}
in the view of Lemma \ref{lem1} (2) and Lemma \ref{lem2}. The last expression is the cofactor expansion of $\tG_{\lambda',f'}$, and thus we obtain the desired formula.

(2) Let $t:=f_1$ and $t':=f_1+1$. Following \cite{MatsumuraFlagged}, we define an equivalence relation $\sim$ on $\FSVT(\lambda,f)$ as follows. Two tableaux $T_1,T_2 \in \FSVT(\lambda,f)$ are equivalent if
\begin{itemize}
\item[(i)] the boxes containing either $t$ or $t'$ in $T_1$ coincide with the ones in $T_2$; 
\item[(ii)] If each box $(1,\lambda_1)$ in $T_1$ and $T_2$ contains $t$, then both of them contain only $t$ or both of them contain entries other than $t$.
\end{itemize}
Let $\scA$ be an equivalence class for $\FSVT(\lambda,f)$. The configuration of $t$ and $t'$ for the tableaux in $\scA$ are depicted as in Figure 1. 
\setlength{\unitlength}{0.5mm}
\begin{center}
\begin{picture}(210,100)\thicklines
\put(0,90){\line(0,-1){50}}\put(0,50){\line(0,-1){50}}\put(190,90){\line(1,0){5}}\put(0,90){\line(1,0){200}}\put(80,80){\line(1,0){120}}
\put(80,70){\line(1,0){90}}\put(80,60){\line(1,0){30}}\put(160,50){\line(1,0){30}}\put(140,30){\line(1,0){20}}\put(0,0){\line(1,0){140}}
\put(180,93){\small{$A_1$}}\put(150,93){\small{$B_1$}}\put(120,83){\small{$A_2$}}\put(90,83){\small{$B_2$}}\put(45,33){\small{$A_k$}}\put(20,33){\small{$B_k$}}
\put(175,83){\small{$t \dots t$}}\put(203,83){\small{$\cdots 1$-st row}}\put(145,83){\small{$t \dots t$}}\put(145,73){\small{$t'\dots t'$}}
\put(85,73){\small{$t \dots t$}}\put(85,63){\small{$t'\dots t'$}}\put(115,73){\small{$*\dots *$}}\put(178,75){\tiny{$m_1$}}
\put(150,65){\tiny{$r_1$}}\put(120,65){\tiny{$m_2$}}\put(90,55){\tiny{$r_2$}}\put(73,48){{$\cdot$}}\put(68,43){{$\cdot$}}\put(63,38){{$\cdot$}}
\put(39,23){\small{$*\dots *$}}\put(15,23){\small{$t\dots t$}}\put(13.5,13){\small{$t'\dots t'$}}\put(160,13){\small{$\cdots k$-th row}}
\put(40,16){\tiny{$m_k$}}\put(19.5,6){\tiny{$r_k$}}\put(10,30){\line(1,0){50}}\put(10,20){\line(1,0){50}}\put(10,10){\line(1,0){27}}\put(10,30){\line(0,-1){20}}
\put(37,30){\line(0,-1){20}}\put(60,30){\line(0,-1){10}}\put(200,90){\line(0,-1){10}}\put(170,90){\line(0,-1){20}}\put(140,90){\line(0,-1){20}}\put(110,80){\line(0,-1){20}}
\put(80,80){\line(0,-1){20}}\put(190,80){\line(0,-1){30}}\put(160,50){\line(0,-1){20}}\put(140,30){\line(0,-1){30}}
\end{picture}\\
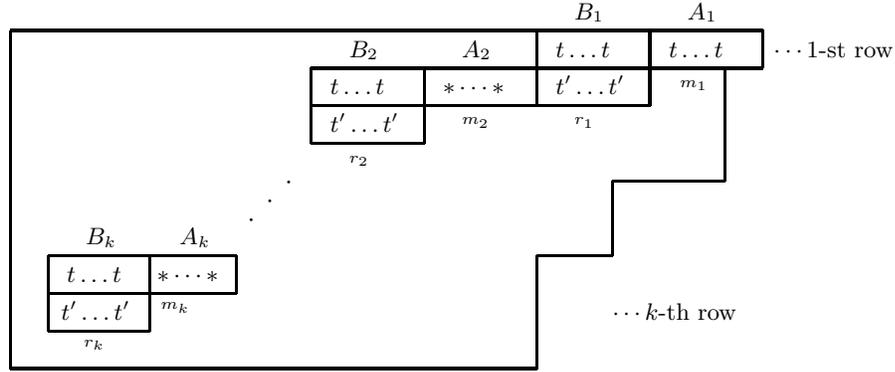
Figure 1.
\end{center}
The one row rectangle $A_1$ on the first row consists of $m_1$ boxes with entries $t$. Each one-row rectangle $A_i\ (2\leq i\leq k)$ with $*$ consists of $m_i$ boxes and each box contains $t$ or $t'$ or both so that the total number of entries $t$ and $t'$ in $A_i$ is $m_i$ or $m_i+1$. Each two-row rectangle $B_j \ (1\leq j\leq k)$ consists of $r_i$ columns with $t$ on the first row and $t'$ on the second. Note that $m_i$ and $r_i$ may be $0$ so that the rectangles in Figure 1 may not be connected. 
Let us write
\[
\sum_{T\in \scA} \beta^{|T|-|\lambda|} [x|b]^T = R(\scA) \left(\prod_{i=1}^k R(A_i)\right) \left(\prod_{j=1}^k R(B_j)\right),
\]
where $R(\scA)$ is the polynomials contributed from the entries other than $t$ and $t'$ and $R(A_i)$ and $R(B_j)$ are the polynomials contributed from the entries $t$ and $t'$ in $A_i$ and $B_j$ respectively. It is easy to see that $R(\scA)$ and $R(B_j)\ (1\leq j\leq k)$ are symmetric in $x_t$ and $x_{t'}$. Moreover, in the view of Lemma \ref{lem3} and Remark \ref{rem2}, $R(A_i) \ (2\leq i \leq k)$ are also symmetric in $x_t$ and $x_{t'}$. We denote the product of these parts symmetric in $x_t$ and $x_{t'}$ by $R'(\scA)$:
\[
R'(\scA):= R(\scA) \left(\prod_{i=2}^k R(A_i)\right) \left(\prod_{j=1}^k R(B_j)\right).
\]
Let $M(T):=\beta^{|T|-|\lambda|} [x|b]^T$. By Lemma \ref{lem1} (2), we have
\[
\pi_t\left(\sum_{T\in \scA} M(T)\right) = \pi_t\left(R(A_1)\right)\cdot R'(\scA), 
\ \ \ \ \ R(A_1)=\prod_{i=1}^{m_1}(x_t\oplus b_{t-1+\lambda_1-m_1+i}).
\]
If $m_1=0$, then $r_1=0$ and $R(A_1)=1$ so that
\begin{equation}\label{eq1}
\pi_t\left(\sum_{T\in \scA} M(T)\right) =-\beta R'(\scA).
\end{equation}
If $m_1=1$, then
\begin{equation}\label{eq2}
\pi_t\left(\sum_{T\in \scA} M(T)\right) =R'(\scA).
\end{equation}
If $m_1\geq2$, then
\begin{eqnarray}
\pi_t\left(\sum_{T\in \scA} M(T)\right)
\!\!\!\!&=&\!\!\!\!\left\{\sum_{v=0}^{m_1-1}\left(\prod_{i=1}^v(x_t\oplus b_{t-1+\lambda_1-m_1+i})\prod_{i=v+1}^{m_i-1}(x_{t'}\oplus b_{t'-1+\lambda_1-m_1+i})\right)\right. \label{eq3}\\
&&\!\!+\left.\beta\sum_{v=1}^{m_1-1}\left(\prod_{i=1}^{v}(x_t\oplus b_{t-1+\lambda_1-m_1+i})\prod_{i=v}^{m_i-1}(x_{t'}\oplus b_{t'-1+\lambda_1-m_1+i})\right)\right\}R'(\scA).\nonumber
\end{eqnarray}

We decompose $\FSVT(\lambda,f)/\!\!\!\sim$ into the subsets $\calF_1,\dots,\calF_4$ of consisting of equivalence classes with configurations satisfying conditions (1) $m_1=0$, (2) $m_1=1$ and $r_1=0$, (3) $m_1=1$ and $r_1\neq 0$, (4) $m_1\geq2$, respectively:
\[
\FSVT(\lambda,f)/\!\!\sim\ =\calF_1\sqcup\calF_2\sqcup\calF_3\sqcup\calF_4.
\]
Let $\widetilde{\calF_2}$ be the subset of $\calF_2$ consisting of equivalence classes with configurations such that the box $(1,\lambda_1)$ contains entries other than $t$. Consider $[\calF_i]:=\{T\ |\ T\in\scA\in\calF_i\}$ and also $[\widetilde{\calF_i}]:=\{T\ |\ T\in\scA\in\widetilde{\calF_i}\}$. First we claim that
\[
\pi_t\left(\sum_{T\in[\calF_1]\sqcup[\widetilde{\calF_2}]} M(T)\right) =0.
\]
There is an obvious bijection from $[\calF_1]$ to $[\widetilde{\calF_2}]$ sending $T$ to the tableau $\widetilde{T}$ obtained by inserting $t$ to the box $(1,\lambda_1)$. This also induces a bijection from $\calF_1$ to $\widetilde{\calF_2}$ ( say, it maps $\scA$ to $\widetilde{\scA}$). Under this bijection, we have
\[
\sum_{\widetilde{T}\in \widetilde{\scA}} M(\widetilde{T}) =\sum_{T\in \scA} M(\widetilde{T}) 
= \beta (x_t\oplus b_{t-1+\lambda_1}) \sum_{T\in \scA} M(T).
\]
Lemma \ref{lem1}, (\ref{eq1}) and (\ref{eq2}) imply that 
\[
\pi_t\left(\sum_{\widetilde{T}\in\widetilde{\scA}}M(\widetilde{T})\right)
=-\beta^2R'(\scA)(x_t\oplus b_{t-1+\lambda_1})
+\beta \sum_{T\in \scA} M(T)
+\beta^2 (x_t\oplus b_{t-1+\lambda_1})R'(\scA) = \beta \sum_{T\in \scA} M(T).
\]
Thus we obtain
\[
\pi_t\left(\sum_{T\in[\calF_1]\sqcup[\widetilde{\calF_2}]}M(T)\right)=\sum_{\scA\in\calF_1}\left\{\pi_t\left(\sum_{T\in\scA}M(T)\right)+\pi_t\left(\sum_{\widetilde{T}\in \widetilde{\scA}}M(\widetilde{T})\right)\right\}=0.
\]
As a consequence, we have
\begin{equation}\label{eq4}
\pi_t\left(G_{\lambda,f}(x|b)\right)
=\pi_t\left(\sum_{T\in[\calF_2]\backslash[\widetilde{\calF_2}]\sqcup[\calF_3]\sqcup[\calF_4]}M(T)\right).
\end{equation}

 Now it remains to show that the right hand side of (\ref{eq4}) coincides with $G_{\lambda',f'}$. Define an equivalence relation $\sim$ in $\FSVT(\lambda',f')$ by the condition (i). For an arbitrary equivalence class $\scA'$ for  $\FSVT(\lambda',f')$, describe its configuration of $t$ and $t'$ as in Figure 2. Similarly to Figure 1, $A_i\ (2\leq i\leq k)$ is a rectangle consisting of $m_i$ boxes with entries $t$, $t'$ or both of them, $B_j\ (1\leq j\leq k)$ is a two-row rectangle with $r_j$ columns with $t$ on the first row and $t'$ on the second. The right-most rectangle $A_1$ has $m_1'$ boxes with entries $t$, $t'$ or both of them. 
\begin{center}
\setlength{\unitlength}{0.5mm}\begin{picture}(210,100)\thicklines\put(0,90){\line(0,-1){50}}\put(0,50){\line(0,-1){50}}\put(190,90){\line(1,0){0}}\put(0,90){\line(1,0){190}}\put(80,80){\line(1,0){110}}\put(80,70){\line(1,0){90}}\put(80,60){\line(1,0){30}}\put(160,50){\line(1,0){30}}\put(140,30){\line(1,0){20}}\put(0,0){\line(1,0){140}}\put(177,93){\small{$A_1$}}\put(150,93){\small{$B_1$}}\put(120,83){\small{$A_2$}}\put(90,83){\small{$B_2$}}\put(45,33){\small{$A_k$}}\put(20,33){\small{$B_k$}}\put(172,83){\small{$* \dots *$}}\put(203,83){\small{$\cdots 1$-st row}}\put(145,83){\small{$t \dots t$}}\put(145,73){\small{$t'\dots t'$}}\put(85,73){\small{$t \dots t$}}\put(85,63){\small{$t'\dots t'$}}\put(115,73){\small{$*\dots *$}}\put(177,75){\tiny{$m'_1$}}\put(150,65){\tiny{$r_1$}}\put(120,65){\tiny{$m_2$}}\put(90,55){\tiny{$r_2$}}\put(73,48){{$\cdot$}}\put(68,43){{$\cdot$}}\put(63,38){{$\cdot$}}\put(40,23){\small{$*\dots *$}}\put(15,23){\small{$t\dots t$}}\put(13.5,13){\small{$t'\dots t'$}}\put(160,13){\small{$\cdots k$-th row}}\put(45,16){\tiny{$m_k$}}\put(19.5,6){\tiny{$r_k$}}\put(10,30){\line(1,0){50}}\put(10,20){\line(1,0){50}}\put(10,10){\line(1,0){27}}\put(10,30){\line(0,-1){20}}\put(37,30){\line(0,-1){20}}\put(60,30){\line(0,-1){10}}\put(190,90){\line(0,-1){10}}\put(170,90){\line(0,-1){20}}\put(140,90){\line(0,-1){20}}\put(110,80){\line(0,-1){20}}\put(80,80){\line(0,-1){20}}\put(190,80){\line(0,-1){30}}\put(160,50){\line(0,-1){20}}\put(140,30){\line(0,-1){30}}\end{picture}\\
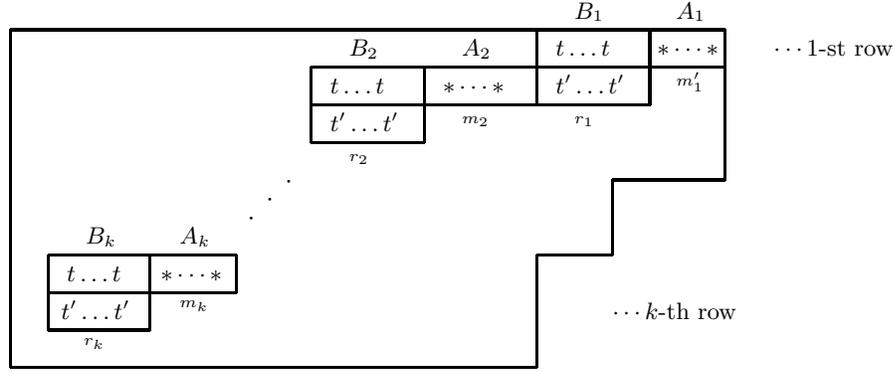
Figure 2.
\end{center}
 
We decompose $\FSVT(\lambda',f')/\!\!\!\sim$ into the subsets $\calF'_2,\calF'_3, \calF'_4$ consisting of equivalence classes with configurations satisfying conditions (2) $m'_1=0$ and $r_1=0$, (3) $m'_1=0$ and $r_1\neq 0$, (4) $m'_1\geq1$, respectively:
\[
\FSVT(\lambda,f)/\!\!\sim\ =\calF'_2\sqcup\calF'_3\sqcup\calF'_4.
\]
As before, let $[\calF'_i]:=\{T'\ |\ T'\in\scA'\in\calF_i'\}$ and then we have
\begin{equation}\label{eq5}
G_{\lambda',f'}(x|b)=\sum_{T'\in[{\calF'_2}]\sqcup[{\calF'_3}]\sqcup[\calF'_4]}M(T')
\end{equation}
There are bijections $\calF_2\backslash\widetilde{\calF_2} \to \calF_2'$, $\calF_3 \to \calF_3'$, and $\calF_4\to \calF_4'$, each of which sends an equivalence class $\scA$ to $\scA'$ whose configuration of $t$ and $t'$ is obtained from the one for $\scA$ by erasing the box $(1,\lambda_1)$ and replacing $t$ with $*$ in the rectangle $A_1$. Under these bijections, we find from (\ref{eq2}) and (\ref{eq3}) that
\[
\pi_t\left(\sum_{T\in \scA} M(T)\right) = \sum_{T'\in \scA'} M(T').
\]
Thus the right hand sides of (\ref{eq4}) and (\ref{eq5}) coincide and we conclude that $\pi_{f_1}G_{\lambda,f}=G_{\lambda',f'}$.
\end{proof}
\begin{thm}\label{mainthm}
For a flagged partition $(\lambda,f)$, we have $G_{\lambda,f}(x|b)=\tG_{\lambda,f}(x|b)$.
\end{thm}
\begin{proof}
We prove the claim by induction on $|f|=f_1+\cdots + f_r$. With the help of Proposition \ref{prop1}, \ref{prop2} and \ref{prop3}, the proof is parallel to the one in \cite{MatsumuraFlagged}. If $|f|=1$, then $\lambda=(\lambda_1)$ and $f=(1)$. Thus it follows from Proposition \ref{prop1}. Suppose that $|f|>1$. We prove in two cases: $f_1=1$ or $f_1>1$. If $f_1=1$, then we can apply Proposition \ref{prop2} to both $\tG_{\lambda,f}$ and $G_{\lambda,f}$. The right hand sides of the resulting formulas coincide by the induction hypothesis, and thus the claim holds. If $f_1>1$, then Proposition \ref{prop3} implies that $\pi_{g_1}\tG_{\mu,g}=\tG_{\lambda,f}$ and $\pi_{g_1}G_{\mu,g}=G_{\lambda,f}$ where $\mu=(\lambda_1+1,\lambda_2,\dots,\lambda_r)$ and $g=(f_1-1,f_2,\dots,f_r)$. The left hand sides of these equalities coincide by the induction hypothesis, and thus the claim holds.
\end{proof}
\section{Vexillary double Grothendieck polynomials}
In this section, we prove that the double Grothendieck polynomials associated to vexillary permutations are in fact factorial Grothendieck polynomials (Theorem \ref{thm1}), giving an alternative proof to the corresponding results in \cite{KnutsonMillerYong} and \cite{HudsonMatsumura2}. Moreover we show that any factorial Grothendieck polynomial can be obtained from a product of certain linear polynomials by applying a sequence of divided difference operators (Theorem \ref{thm2}).

The {\it double Grothendieck polynomials} were introduced by Lascoux and Sch\"utzenberger \cite{Lascoux1}, \cite{LascouxSchutzenberger3}. 
For any permutation $w\in S_n$, we define the associated double Grothendieck polynomial $\frakG_w=\frakG_w(x|b)$ as follows. Let $w_0$ be the longest element of the symmetric group $S_n$. We set
\[
\frakG_{w_{0}}=\prod_{i+j\leq n}(x_i\oplus b_j).
\]
For an element $w\in S_n$ such that $\ell(w)<\ell(w_0)$, we can choose a simple reflection $s_i \in S_n$ such that  $\ell(ws_i)=\ell(w)+1$. Here $\ell(w)$ is the length of $w$. We then define
\[
\frakG_w=\pi_i(\frakG_{ws_i}).
\]
The polynomial $\frakG_w$ is defined independently from the choice of such $s_i$, since the divided difference operators satisfy the Coxeter relations. In this point of view, we can write $\frakG_w=\pi_v \frakG_{w_0}$ with $v=w_0w$ and $\pi_v=\pi_{i_k}\cdots \pi_{i_1}$ where $v=s_{i_1}\cdots s_{i_k}$ is a reduced expression.

A permutation $w\in S_n$ is called {\it vexillary} if it is $2143$-avoiding, {\it i.e.} there is no $a<b<c<d$ such that $w(b)<w(a)<w(d)<w(c)$. We briefly recall how to obtain a flagged partition from a vexillary permutation. We follow \cite{FlagsFulton} and \cite{KnutsonMillerYong} ({\it cf.} \cite{HudsonMatsumura2}). For each $w\in S_n$, let $r_w$ be the \emph{rank function} of $w \in S_n$ defined by $r_w(p,q) := \sharp\{ i\leq p  |\ w(i) \leq q\}$ for $1\leq p,q \leq n$. The {\it diagram} $D(w)$ of $w$ is defined as
\[ 
D(w):=\{(p,q)\in \{1,\dots,n\}\times \{1,\dots,n\} \ |\ w(p) >q, \ \mbox{and} \ w^{-1}(q)>p\}.
\]
The \emph{essential set} $\Ess(w)$ is the subset of $D(w)$ given by
\[
\Ess(w):=\{(p,q) \ |\ (p+1,q), (p,q+1) \not\in D(w)\}.
\]
If $w$ is vexillary, we can choose a {\it flagging set} of $w$ (cf. \cite{HudsonMatsumura2}), which is a subset $\{(p_i,q_i),i=1,\dots,r\}$ of $\{1,\dots,n\}\times \{1,\dots,n\}$ containing $\Ess(w)$ and satisfying 
\begin{eqnarray}
&&p_1\leq p_2 \leq \cdots \leq p_r,  \ \ \ q_1\geq q_2 \geq \cdots \geq q_r, \label{flagging1}\\
&&p_i - r_w(p_i,q_i) = i, \ \ \ \forall i=1,\dots,r. \label{flagging2} 
\end{eqnarray}
An associated flagged partition $(\lambda(w),f(w))$ of length $r$ is given by setting $f_i(w):=p_i$ and $\lambda(w)_i= q_i-p_i +i$ for $i=1,\dots,r$. We remark that the set $\FSVT(\lambda(w),f(w))$ depends only on $w$ but not on a choice of a flagging set. 
\begin{exm}
A permutation $w\in S_n$ is {\it dominant} if $D(w)$ is a Young diagram of a partition and the values of $r_w$ on $D(w)$ are zero. Such permutation is vexillary, and in this case, $\lambda(w)$ is the partition whose Young diagram is $D(w)$ and the flagging is $f(w)=(1,2,\cdots,r)$ where $r$ is the length of $\lambda(w)$. 
\end{exm}

The following theorem was obtained in \cite{KnutsonMillerYong} and \cite{HudsonMatsumura2}. We given an alternative proof.
\begin{thm}\label{thm1}
If $w\in S_n$ is vexillary, then we have $\frakG_w=G_{\lambda(w),f(w)} (=\tG_{\lambda(w),f(w)})$.
\end{thm}
\begin{proof}
We closely follow the argument in \cite{Wachs} and \cite{MatsumuraFlagged}. We prove the claim by induction on $(n, \ell(w_0)-\ell(w))$ with the lexicographic order.  For the longest element $w_0\in S_n$, we have $\lambda(w_0)=(n-1,n-2,\cdots,2,1)$ and $f(w_0)=(1,2,\cdots,n-1)$. Thus
\[
G_{\lambda(w_0),f(w_0)}=\prod_{i+j\leq n}(x_i\oplus b_j)= \frakG_{w_{0}}
\]
Suppose that $w\in S_0$ and $w\not=w_0, \id$. Let $d$ be the leftmost descent of $w$, {\it i.e.} $d$ is the smallest number such that $w(d)>w(d+1)$.

{\bf Case 1} ($d>1$). We observe that $w'=ws_{d-1}$ is vexillary and $\ell(w')>\ell(w)$. By the induction hypothesis and Proposition \ref{prop3}, we have
\[
\frakG_w=\pi_{d-1}\frakG_{w'}=\pi_{d-1}G_{\lambda(w'),f(w')}=G_{\lambda(w),f(w)}.
\]

{\bf Case 2} ($d=1$ and $m:=w(1)<n$). We observe that $w'=s_mw$ is vexillary and $l(w')=l(w)+1$.
Consider the dominant permutation
\[
u=(m+1,m,m+2,m+3,\cdots,n,m-1,m-2,\cdots,1).
\]
Since $w$ is vexillary, all numbers greater than $m$ appear in ascending order in $w$. Thus we can find integers $i_1,\dots, i_k$ greater than $1$, satisfying $w=us_1s_{i_k} \cdots s_{i_1}$ and $\ell(w)+1 + l = \ell(u)$. Together with $s_mu=us_1$, we have $\frakG_w=\pi_{i_1}\cdots\pi_{i_k}\pi_1\frakG_u$ and $\frakG_{w'}=\pi_{i_1}\cdots\pi_{i_k}\frakG_u$.
On the other hand, we see that $(\lambda(u)_1,\lambda(u)_2) = (m,m-1)$ and $(f(u)_1,f(u)_2)=(1,2)$ so that $\pi_1\frakG_u=(x_1\oplus b_m)^{-1}\frakG_u$. Since $i_1,\dots,i_k\geq 2$, we find that, by the induction hypothesis, 
\[
\frakG_w=\pi_{i_1}\cdots\pi_{i_k} \left((x_1\oplus b_m)^{-1}\frakG_u\right)=(x_1\oplus b_m)^{-1}\frakG'_w=(x_1\oplus b_m)^{-1}G_{\lambda(w'),f(w')}.
\]
Since $f(w')_1=f(w)_1=1$ and the flagged partitions of $w$ and $w'$ are identical except $m=\lambda(w')_1=\lambda(w)_1+1$, we have $G_{\lambda(w),f(w)}=(x_1 \oplus b_m)^{-1}G_{\lambda(w'),f(w')}$. This shows that $\frakG_w=G_{\lambda(w),f(w)}$.

{\bf Case 3} ($d=1$ and $w(1)=n$). We can find a reduced expression $s_{i_1}\cdots s_{i_k}$ of $w_0w$ where $i_1,\dots, i_k \geq 2$. Let $w'=(w(2),\dots, w(n))\in S_{n-1}$. Then we have
\[
\frakG_w=\prod_{i=1}^{n-1}(x_1\oplus b_i)\cdot \pi_{i_k}\pi_{i_{k-1}}\cdots\pi_{i_1}\left(\prod_{i+j\leq n-1}(x_{i+1}\oplus b_j)\right) = \prod_{i=1}^{n-1}(x_1\oplus b_i)\cdot \left(\frakG_{w'}\right)^{\star}
\]
where $\star$ replaces each $x_i$ by $x_{i+1}$. By the induction hypothesis and Proposition \ref{prop1}, we have
\[
\frakG_w = \calG_{n-1}^{[1,n-1]} \left(\frakG_{\lambda(w'),f(w')}\right)^{\star}
\]
Since $f(w)_1=1$ and $(f(w')_1,\cdots,f(w')_{n-1})=(f(w')_2-1,\cdots,f(w')_n-1)$, Proposition \ref{prop2} implies the claim.
\end{proof}
\begin{thm}\label{thm2}
Let $(\lambda,f)$ be a flagged partition of length $r$. Then we have
\[
G_{\lambda,f}=\pi_w\left(\prod_{i=1}^r\prod_{j=1}^{a_i}(x_i\oplus b_j)\right)
\]
where $a_i=\lambda_i+f_i-i$ for $i=1,\dots,r$ and 
\[
w=(s_rs_{r+1}\cdots s_{f_r-1})\cdots (s_2s_3\cdots s_{f_2}-1)(s_1s_2\cdots s_{f_1}-1).
\]
\end{thm}
\begin{proof}
We prove by induction on $|f|$. If $|f|=1$, we see that $G_{\lambda,f}=x_1\oplus b_1$ and $w=\id$ so that the claim is trivial. Suppose that $|f|>1$. If $f_1=1$, let $\lambda'=(\lambda_2,\cdots,\lambda_r)$ and $f'=(f_2-1,\cdots,f_r-1)$. By Proposition \ref{prop2}, we have $G_{\lambda,f}=(G_{\lambda',f'})^{\star}\prod_{j=1}^{\lambda_1}(x_1\oplus b_j)$. By the induction hypothesis, we can write $(G_{\lambda',f'})^{\star} = \pi_w\left(\prod_{i=2}^r\prod_{j=1}^{a_i}(x_i\oplus b_j)\right)$. Since $s_1$ doesn't appear in the reduced expression of $w$, we obtain the desired formula. If $f_1>1$, since $\lambda'=(\lambda_1+1,\lambda_2,\cdots,\lambda_r),f'=(f_1-1,f_2,\cdots,f_r)$, Proposition \ref{prop3} and the induction hypothesis imply the claim:
\[
G_{\lambda,f}=\pi_{f_1-1}G_{\lambda',f'}=\pi_{f_1-1}\pi_{ws_{f_1-1}}\left(\prod_{i=1}^r\prod_{j=1}^{a_i}(x_i\oplus b_j)\right)=\pi_w\left(\prod_{i=1}^r\prod_{j=1}^{a_i}(x_i\oplus b_j)\right).
\]
This completes the proof.
\end{proof}

\bibliography{references}{}
\bibliographystyle{acm}

%

\end{document}